\documentclass[11pt,leqno]{article}

\usepackage[active]{srcltx}

\usepackage{amsfonts,latexsym,amsmath,amssymb,amsthm}
\usepackage{fullpage}
\newtheorem{theorem}{Theorem}[section]
\newtheorem{lemma}[theorem]{Lemma}

\newtheorem{proposition}[theorem]{Proposition}
\newtheorem{corollary}[theorem]{Corollary}
\newtheorem{remark}[theorem]{Remark}

\theoremstyle{definition}
\newtheorem{definition}[theorem]{Definition}

\theoremstyle{remark}

\newtheorem*{note*}{Note}

\numberwithin{equation}{section}

\makeatletter
\newcommand{\rank}{\mathop{\operator@font rank}}
\newcommand{\conv}{\mathop{\operator@font conv}}
\newcommand{\vol}{\mathop{\operator@font vol}}
\newcommand{\onetagright}{\tagsleft@false}
\makeatother

\newcommand{\ls}{\leqslant}
\newcommand{\gr}{\geqslant}

\usepackage{color}

\usepackage{bbm}

\usepackage{hyperref}

\begin{document}
\small

\title{\bf A note on norms of signed sums of vectors}

\author{Giorgos Chasapis and Nikos Skarmogiannis}

\date{}

\maketitle

\begin{abstract}
\footnotesize Our starting point is an improved version of a result of D.~Hajela related to a question of Koml\'{o}s: we show that if $f(n)$ is a function
such that $\lim\limits_{n\to\infty }f(n)=\infty $ and $f(n)=o(n)$, there exists $n_0=n_0(f)$ such that for every $n\gr n_0$ and
any $S\subseteq \{-1,1\}^n$ with cardinality $|S|\ls 2^{n/f(n)}$ one can find orthonormal vectors $x_1,\ldots ,x_n\in {\mathbb R}^n$ that satisfy
$$\|\epsilon_1x_1+\cdots +\epsilon_nx_n\|_{\infty }\gr c\sqrt{\log f(n)}$$
for all $(\epsilon_1,\ldots ,\epsilon_n)\in S$. We obtain analogous results in the case where
$x_1,\ldots ,x_n$ are independent random points uniformly distributed in the Euclidean unit ball $B_2^n$ or any
symmetric convex body, and the $\ell_{\infty }^n$-norm is replaced by an arbitrary norm on ${\mathbb R}^n$.
\end{abstract}

\section{Introduction}

Given a pair $K,D$ of symmetric convex bodies in ${\mathbb R}^n$, the parameter $\beta (K,D)$ is the smallest $r>0$ such that
for any $x_1,\ldots,x_n \in K$ there exist signs $\epsilon_1, \ldots,\epsilon_n\in \{-1,1\}$ for which
\begin{equation*}\epsilon_1x_1+\cdots+\epsilon_nx_n \in rD.\end{equation*}
A general lower bound for $\beta (K,D)$ was proved by Banaszczyk; in \cite{Banaszczyk-1993} he showed that if $K$ and $D$ are
two symmetric convex bodies in ${\mathbb R}^n$ then
\begin{equation}\label{eq:ban-lower}\beta (K,D)\gr c\sqrt{n}(|K|/|D|)^{1/n}\end{equation}
for an absolute constant $c>0$, where $|K|$ denotes the volume of $K$. In what follows we write $B_p^n$ for the unit ball of $\ell_p^n=({\mathbb R}^n,\|\cdot\|_p)$, $1\ls p\ls\infty $.
A well-known theorem of Spencer \cite{Spencer-1985} states that $\beta (B_{\infty }^n,B_{\infty }^n)\ls c\sqrt{n}$, where $c>0$ is an absolute constant
(the same result was proved independently by Gluskin in \cite{Gluskin-1989}): there exists an absolute constant $c>0$ such that
for any $n\gr 1$ and $x_1,\ldots ,x_n\in {\mathbb R}^n$ with $\|x_i\|_{\infty }\ls 1$, we may find $\epsilon_1,\ldots ,\epsilon_n\in\{ -1,1\}$
such that
\begin{equation}\label{eq:signed-1}\|\epsilon_1x_1+\cdots +\epsilon_nx_n\|_{\infty }\ls c\sqrt{n}.\end{equation}
From \eqref{eq:ban-lower} one readily sees that this result is optimal up to absolute constants.
A well-known question of Koml\'{o}s (see \cite{Spencer-1986} and \cite{Spencer-book})
asks if the sequence $\beta(B_2^n,B_{\infty }^n)$ is bounded. Since $B_{\infty }^n\subseteq \sqrt{n}B_2^n$,
a positive answer to this question immediately implies Spencer's bound. The best known estimate
is due to Banaszczyk \cite{Banaszczyk-1998}: there exists an absolute constant $c>0$ such that for any
$n\gr 1$ and $x_1,\ldots ,x_n\in {\mathbb R}^n$ with $\|x_i\|_2\ls 1$ one may find $\epsilon_1,\ldots ,\epsilon_n\in\{ -1,1\}$
such that
\begin{equation}\label{eq:signed-2}\|\epsilon_1x_1+\cdots +\epsilon_nx_n\|_{\infty }\ls c\sqrt{\log n}.\end{equation}
In fact, Banaszczyk proved a more general theorem: if $K$ is a convex body in ${\mathbb R}^n$ with
Gaussian measure $\gamma_n(K)\gr 1/2$ then $\beta (B_2^n,K)\ls C$, where $C>0$ is an absolute constant; this
implies \eqref{eq:signed-2} because $\gamma_n(rB_{\infty }^n)\gr 1/2$ for all $r\gr c\sqrt{\log n}$.
While the method of \cite{Banaszczyk-1998} is non-constructive, an algorithmic proof of the $O(\sqrt{\log n})$ bound in the above statement was recently given by Bansal, Dadush and Garg \cite{BDG}.

The starting point of this note is a result of Hajela \cite{Hajela-1988} in the direction of providing a negative
answer to the question of Koml\'{o}s.

\begin{theorem}[Hajela]\label{th:hajela}Let $f(n)$ be a function such that $\lim\limits_{n\to\infty }f(n)=\infty $ and $f(n)=o(n)$.
For any $0<\lambda <\frac{1}{2}$ there exists $n_0=n_0(f,\lambda )$ such that for every $n\gr n_0$ and any $S\subseteq \{-1,1\}^n$
with cardinality $|S|\ls 2^{n/f(n)}$ one can find orthonormal vectors $x_1,\ldots ,x_n\in {\mathbb R}^n$ that satisfy
$$\|\epsilon_1x_1+\cdots +\epsilon_nx_n\|_{\infty }\gr\exp\left (\frac{\lambda \log\log f(n)}{\log\log\log f(n)}\right )$$
for all $(\epsilon_1,\ldots ,\epsilon_n)\in S$.
\end{theorem}

In fact, Hajela conjectures in \cite{Hajela-1988} that the question of Koml\'{o}s has a negative answer and that the estimate
\eqref{eq:signed-2} that was later obtained by Banaszczyk should be optimal. Our first result is an improved version of Theorem \ref{th:hajela}.

\begin{theorem}\label{intro-thm.1}
There exists an absolute constant $c\in (0,1)$ such that the following holds true:
For every $n\gr 1$ and $\frac{1}{n}<\delta <1$, and for any $S\subseteq \{-1,1\}^n$ with $|S|\ls 2^{\delta n}$, there exist orthonormal
vectors $x_1,\ldots,x_n$ in $\mathbb{R}^n$ such that
\begin{equation*}\left\|\sum_{i=1}^n \epsilon_ix_i\right\|_\infty \gr c\sqrt{\log (e/\delta )}\end{equation*}
for all $(\epsilon_1,\ldots,\epsilon_n)\in S$.
\end{theorem}

Theorem \ref{intro-thm.1} implies a stronger version of Hajela's theorem. Let $f(n)$ be a function such that $\lim\limits_{n\to\infty }f(n)=\infty $ and $f(n)=o(n)$.
Note that if we let $\delta =\delta (f,n) =ef(n)^{-1}$ in Theorem \ref{intro-thm.1} then $\frac{1}{n}<\delta <1$ for large enough $n$, resulting to the lower bound
\begin{equation*}\left\|\sum_{i=1}^n \epsilon_i x_i \right\|_\infty \gr c\sqrt{\log f(n)},\end{equation*}
which improves upon the estimate of Theorem \ref{th:hajela}. In the proof of Theorem \ref{intro-thm.1}, presented in Section \ref{sec:hajela.ext} below,
initially we follow the idea of Hajela: the vectors $x_1,\ldots,x_n$ are obtained from a random rotation of the standard basis $e_1,\ldots,e_n$ of $\mathbb{R}^n$. The improvement
on the estimates is due to Lemma \ref{lem:hajela-1} which provides a stronger small-ball probability estimate for the
$\|\cdot\|_{\infty }$-norm on the sphere.

The above method to lower-bound the $\ell_\infty$-norm of a signed sum of vectors has obvious limitations. Namely, one should consider a subset $S\subseteq \{-1,1\}^n$ of cardinality $2^{o(n)}$ for Theorem \ref{intro-thm.1} to obtain a lower bound of order greater than the one in \eqref{eq:ban-lower}. Thus, this line of thinking by itself does not seem adequate to provide a negative answer to the question of Koml\'{o}s. Nevertheless, we find the link between small ball probability estimates and the norm of signed sums of vectors interesting in its own right, so in Section \ref{sec:random.rot} we further explore this phenomenon in several aspects. Initially, the vectors $x_1,\ldots x_n$ are assumed to satisfy a lower bound of the form $\|\sum_{i=1}^n \epsilon_ix_i\|_2\gr c\sqrt{n}$ for all choices of signs $\epsilon_i=\pm 1$, and the $\ell_{\infty }^n$-norm is replaced by an arbitrary norm on ${\mathbb R}^n$. In particular, given a symmetric convex body $D$ in $\mathbb{R}^n$, we will prove that for $x_1,\ldots x_n$ as above and for any ``big'' subset $S\subseteq\{-1,1\}^n$, the $D$-norm of $\sum_{i=1}^n\epsilon_iUx_i$ is ``large'' for every $(\epsilon_1,\ldots,\epsilon_n)\in S$, with overwhelming probability with respect to $U\in O(n)$. This ``largeness'' is determined by the Gaussian measure of dilates of $D$; to make this more precise, we give the following definition.
\begin{definition}
Given a symmetric convex body $D$ in $\mathbb{R}^n$, and $\delta\in (0,1)$, let
\[
t_{D,\delta} := \max\{t>0 : \gamma_n(2t\cdot m(D)D) \ls (2^\delta e)^{-n}\},
\]
where $m(D)$ is the median of $\|\cdot\|_D$ with
respect to the standard Gaussian measure $\gamma_n$ on $\mathbb{R}^n$. In the case that $D$ is the unit ball $B_p^n$ of some $\ell_p^n$, $p\in [1,\infty]$, we abbreviate $t_{p,\delta} := t_{B_p^n,\delta}$.
\end{definition}

Note that, for any symmetric convex body $D$ in $\mathbb{R}^n$ and $\delta\in(0,1)$, $t_{D,\delta}$ satisfies the bounds
\begin{equation}\label{eq.t_D.bounds}
c_1|D|^{-1/n}\ls t_{D,\delta}m(D) \ls \frac{1}{2}m(D),
\end{equation}
for some absolute constant $c_1>0$ (for completeness, we provide a short justification of \eqref{eq.t_D.bounds} in Section \ref{subsec:3.1}). Although the definition of $t_{D,\delta}$ might seem somewhat artificial at first sight, we trust that the idea behind it will become clear in the sequel (see the comments after Theorem \ref{general-thm}).

Using the notation introduced above, our first result is the following statement.

\begin{theorem}\label{intro-thm.2}Let $D$ be a symmetric convex body in ${\mathbb R}^n$ and $\delta\in (0,1)$. For any  $\tau >0$, any
$n$-tuple of vectors $x_1,\ldots ,x_n$ with $\min\limits_{\epsilon_i=\pm 1}\left\|\sum_{i=1}^n\epsilon_ix_i\right\|_2\gr \tau\sqrt{n}$ and any $S\subseteq\{-1,1\}^n$ with $|S|\ls 2^{\delta n}$, there is some ${\cal U}\subseteq O(n)$ with $\nu_n({\cal U})\gr 1-e^{-n}$ such that for every $U\in{\cal U}$,
$$\left\|\sum_{i=1}^n\epsilon_iU(x_i)\right\|_D\gr \tau t_{D,\delta}m(D)$$
holds for all $\epsilon =(\epsilon_1,\ldots ,\epsilon_n)\in S$.
\end{theorem}

We then consider the case where $x_1,\ldots,x_n$ are points in an arbitrary convex body $K$ in $\mathbb{R}^n$.
We observe that an alternative proof of \eqref{eq:ban-lower} can be deduced from a more general result of
Gluskin and V.~Milman in \cite{Gluskin-Milman-2004}. Given a star body $D$ in ${\mathbb R}^n$ with $0\in {\rm int}(D)$
and some measurable subsets $V_1,\ldots ,V_m$ of ${\mathbb R}^n$, we denote for the rest of this article by $\mathbb{P}$ the probability taken with respect to the product of the uniform probability measures $\mu_i(A)=\frac{|A\cap V_i|}{|V_i|}$. Using this notation, the aforementioned result of Gluskin and V. Milman can be stated as follows: Let $V_1,\ldots,V_m, D$ be as above, and such that $|D|=|V_1|=\cdots =|V_m|$. Then for any $\lambda_1,\ldots ,\lambda_m\in {\mathbb R}$
and any $0<t<1$ one has
\[
\mathbb{P}\left( (x_i)_{i=1}^m, x_i\in V_i: \left\|\sum_{i=1}^m \lambda_ix_i \right\|_D \ls t\Big (\sum_{i=1}^m\lambda_i^2\Big )^{1/2}\right) \ls
\left (te^{\frac{1-t^2}{2}}\right )^n.
\]
The proof of this estimate in \cite{Gluskin-Milman-2004} is based on (the sharp form of) the multivariate Young inequality, see \cite{Bec} and \cite{BrasLieb}.

As a next step, using Theorem \ref{intro-thm.2} we obtain the following variant of the result of Gluskin and Milman, in the case that $V_i=B_2^n$ for all $i$.

\begin{theorem}\label{general-thm}
Let $\delta\in (0,1)$, $D$ be a symmetric convex body in $\mathbb{R}^n$ and $S\subseteq \{-1,1\}^n$ with $|S|\ls 2^{\delta n}$. Then
\[
\mathbb{P}\left( (x_i)_{i=1}^n\subseteq B_2^n : \left\|\sum_{i=1}^n \epsilon_ix_i\right\|_D \ls \frac{1}{10}t_{D,\delta}m(D) \hbox{, for some } \epsilon\in S\right)
\ls 3e^{-n}.
\]
\end{theorem}

This theorem may be viewed as an extension of Hajela's result: the $\ell_{\infty }^n$-norm is replaced by an arbitrary norm
on ${\mathbb R}^n$ and the statement holds for a random choice of vectors in the Euclidean unit ball. In this context, the role of the parameter $t_{D,\delta}$ in the statement of our results becomes now more transparent:
Since, by \eqref{eq.t_D.bounds},
$$t_{D,\delta }m(D)\simeq t_{D,\delta }\sqrt{n}M(D)\simeq t_{D,\delta }M(D){\rm vrad}(D) \,|D|^{-1/n},$$
where $M(D) = \int_{S^{n-1}}\|x\|_D\,d\sigma (x)$ and ${\rm vrad}(D)=(|D|/|B_2^n|)^{1/n}$, and since $M(D){\rm vrad}(D)\gr 1$
(this is a simple consequence of H\"{o}lder's inequality), we see that Theorem \ref{general-thm} provides information
stronger than the one from \eqref{eq:ban-lower} provided that $M(D){\rm vrad}(D)\gg 1$ and/or $t_{D,\delta }\simeq 1$. This is
the case for the $\|\cdot\|_{\infty }$-norm and it would be interesting to provide further examples with sharp estimates.

The function $t\mapsto\gamma_n(2t\cdot m(D)D)$ that appears in the definition of $t_{D,\delta}$ above has been studied
in \cite{Klartag-Vershynin-2007}, \cite{Latala-Oleszkiewicz-2005} and \cite{Paouris-Valettas-2017} and elsewhere.
A sample of estimates is the following:
\begin{itemize}
\item In \cite{Klartag-Vershynin-2007} it is proved that for every $0<t<\frac{1}{2}$ one has
$$\gamma_n(\{x:\|x\|_D\ls t\cdot m(D)\})\ls\frac{1}{2}t^{d(D)},$$ where
$$d(D)=\min\left\{ n,-\log\gamma_n\left (\frac{m(D)}{2}D\right )\right\},$$
(see \cite[Theorem~3.1]{Paouris-Valettas-2017} for this exact
formulation of the result).
\item In \cite{Latala-Oleszkiewicz-2005} it is proved that if $\gamma_n(D)\ls\frac{1}{2}$ then for every $0<t<\frac{1}{2}$ one has
$\gamma_n(tD)\ls (2t)^{r^2(D)/4}\gamma_n(D)$ where $r(D)$ is the inradius of $D$.
\item In \cite{Paouris-Valettas-2017} it is proved that for every $0<t<\frac{1}{2}$ one has $\gamma_n(\{x:\|x\|_D\ls t\cdot m(D)\})\ls\frac{1}{2}t^{c/\beta (D)}$, where
$$\beta (D)=\frac{{\rm Var}_{\gamma_n}(\|\cdot\|_D)}{M^2(D)}$$
and $c>0$ is an absolute constant.
\end{itemize}
In this note we discuss the estimates that can be derived from the above in special instances, as for example when $D$ is an $\ell_p^n$ ball.

Finally, the argument in the proof of Theorem \ref{general-thm} can be further
generalized to the case where $x_1,\ldots ,x_n$ are independent random points chosen uniformly from any symmetric convex body.

\begin{theorem}\label{intro-thm.5}
There exists an absolute constant $c>0$ such that the following holds: Let $\delta\in (0,1)$, $D$ be a symmetric convex body in $\mathbb{R}^n$ and $S\subseteq \{-1,1\}^n$ with $|S|\ls 2^{\delta n}$. For any symmetric convex body $K$ in ${\mathbb R}^n$ with $|K|=|B_2^n|$ we may find ${\cal U}\subseteq O(n)$ with $\nu_n ({\cal U})>1-3e^{-n/2}$
such that, for any $U\in {\cal U}$,
\begin{equation*}\mathbb{P}\left( (z_i)_{i=1}^n \subseteq U(K)\times\cdots\times U(K) : \left\| \sum_{i=1}^n \epsilon_iz_i \right\|_D \ls ct_{D,\delta}m(D) \hbox{,
for some } \epsilon\in S \right)
\ls e^{-n/2}.\end{equation*}
\end{theorem}

\section{An improved version of Hajela's result}\label{sec:hajela.ext}

In this section we present the proof of Theorem \ref{intro-thm.1}. In what follows, $e_1,\ldots,e_n$ is the standard basis of $\mathbb{R}^n$. We denote by $S^{n-1}$ the Euclidean sphere in $\mathbb{R}^n$, and by $\sigma$ the unique rotationally invariant probability measure on $S^{n-1}$. Recall that $\sigma$ can be defined via the Haar probability measure $\nu_n$ on the orthogonal group $O(n)$ as follows: For every measurable $A\subseteq S^{n-1}$ we have
\[
\sigma(A) = \nu_n(\{U\in O(n) : U(x) \in A\}),
\]
for an arbitrary $x\in S^{n-1}$. We need a small ball probability estimate for the $\ell_{\infty }^n$-norm, which is given in the next lemma. It provides a bound whose behaviour for small values of $t$ will play a crucial role in the sequel;
in the case of $\|\cdot\|_{\infty }$ this type of behaviour has been observed in various works (see e.g. \cite{Szankowski-1974} and \cite{VMilman-Schechtman-1995}).
We provide a direct and short proof of the exact statement that we need.

\begin{lemma}\label{lem:hajela-1}There exists an absolute constant $c\in (0,1)$ such that, for any $n\gr 1$ and $\delta\in \left (\frac{1}{n},\frac{1}{4}\right )$,
\begin{equation*}\sigma \left(\left\{ \theta \in S^{n-1} : \|\theta\|_\infty \ls \frac{c\sqrt{\log (e/\delta )}}{\sqrt{n}}\right\} \right)
< 2^{-\delta n}.\end{equation*}
\end{lemma}

\begin{proof}
We use the fact (see \cite{Klartag-Vershynin-2007} for its simple proof) that if $A$ is a symmetric convex body in ${\mathbb R}^n$ then
\begin{equation}\label{eq.sigma.to.gamma}\sigma (S^{n-1}\cap A)\ls 2\gamma_n(2\sqrt{n}A),\end{equation}
to write
\begin{align*}\sigma \left(\left\{ \theta \in S^{n-1} : \|\theta\|_\infty \ls \frac{s}{2\sqrt{n}}\right\} \right)
&\ls 2\gamma_n\left(\left\{ x\in {\mathbb R}^n: \|x\|_\infty \ls s\right\} \right) \\
&= 2(1-2(1-\Phi(s)))^n \ls 2\exp(-2n(1-\Phi(s))),
\end{align*}
where, as usual, $\Phi$ stands for the cumulative distribution function of the standard normal distribution, that is, $\Phi(x) = (2\pi)^{-1/2}\int_{-\infty}^x e^{-t^2/2}\,dt$. Using the inequality $(t+s)^2\ls 2(t^2+s^2)$, we have that, for any $s>0$,
\[
1-\Phi(s) = \frac{1}{\sqrt{2\pi}}\int_0^\infty e^{-\frac{(t+s)^2}{2}} \,dt \gr \frac{e^{-s^2}}{\sqrt{2\pi}}\int_0^\infty e^{-t^2} \,dt = \frac{e^{-s^2}}{2\sqrt{2}}.
\]
It follows that
\[
\sigma \left(\left\{ \theta \in S^{n-1} : \|\theta\|_\infty \ls \frac{s}{2\sqrt{n}}\right\} \right) \ls 2\exp\left(-\frac{n}{\sqrt{2}}e^{-s^2}\right).
\]
Finally, if $\delta\in\left(\frac{1}{n},\frac{1}{4}\right)$ then choosing $s=\frac{1}{2}\sqrt{\log\left(\frac{e}{\delta} \right)}$ we get
\[
2\exp\left(-\frac{n}{\sqrt{2}}e^{-s^2}\right) = 2\exp\left(-\frac{n}{\sqrt{2}}\left(\frac{\delta}{e}\right)^{1/4}\right) < 2^{-\delta n},
\]
which proves the desired statement.
\end{proof}

\medskip

\begin{proof}[Proof of Theorem \ref{intro-thm.1}]
The idea of the proof is the same as in Hajela's paper. We consider orthonormal systems obtained as random rotations of the standard basis
$e_1,\ldots,e_n$ of $\mathbb{R}^n$.
Let $\delta\in (n^{-1},1/4)$. Since, for any $\epsilon = (\epsilon_1,\ldots,\epsilon_n) \in \{-1,1\}^n$, $n^{-1/2}\sum_{i=1}^n \epsilon_i e_i \in S^{n-1}$, by the definition of the measure $\sigma$ on $S^{n-1}$ it follows,
taking $\alpha =c\sqrt{\log (e/\delta )}$ where $c $ is the constant in Lemma \ref{lem:hajela-1}, that
\begin{align*}
&\nu_n\left(\left\{ U\in O(n) : \left\|U\left(\sum_{i=1}^n \epsilon_i e_i\right)\right\|_\infty \ls c \sqrt{\log (e/\delta )}\right\} \right) \\
&\hspace*{3cm}= \nu_n\left(\left\{ U\in O(n) : \left\|U\left(\frac{\sum_{i=1}^n \epsilon_ie_i}{\sqrt{n}}\right)\right\|_\infty \ls \frac{c \sqrt{\log (e/\delta )}}{\sqrt{n}}\right\} \right)\\
&\hspace*{3cm}= \sigma \left(\left\{ \theta \in S^{n-1} : \|\theta\|_\infty \ls \frac{c\sqrt{\log (e/\delta )}}{\sqrt{n}}\right\} \right).
\end{align*}
Applying Lemma \ref{lem:hajela-1} we get
\[
\sigma \left(\left\{ \theta \in S^{n-1} : \|\theta\|_\infty \ls \frac{c\sqrt{\log (e/\delta )}}{\sqrt{n}}\right)\right )
< 2^{-\delta n}.
\]
This shows that
\begin{equation}\label{eq.1}
\nu_n\left(\left\{ U\in O(n) : \left\|U\left(\sum_{i=1}^n \epsilon_i e_i\right)\right\|_\infty \ls c\sqrt{\log (e/\delta )}\right\} \right) < 2^{-\delta n}.
\end{equation}
Now let $S\subseteq \{-1,1\}^n$ with $S\ls 2^{\delta n}$. Then, by the union bound and \eqref{eq.1},
\begin{align*}
\nu_n &\left( \left\{ U\in O(n) : \left\| \sum_{i=1}^n \epsilon_i U(e_i)\right\|_\infty \gr c\sqrt{\log (e/\delta )} \hbox{\,, for every } \epsilon\in S \right\} \right) =\\
                                 &\hspace{80pt}= 1 - \nu_n\left(\left\{ U\in O(n) : \left\|\sum_{i=1}^n \epsilon_i U(e_i)\right\|_\infty \ls c\sqrt{\log (e/\delta )} \hbox{\,, for some } \epsilon\in S\right\}\right)\\
                                  &\hspace{80pt}\gr 1 - \sum_{\epsilon \in S} \nu_n\left(\left\{ U\in O(n) : \left\|\sum_{i=1}^n \epsilon_i U(e_i)\right\|_\infty \ls c\sqrt{\log (e/\delta )} \right\}\right) \\
                                  &\hspace{80pt}> 1 - |S|\cdot2^{-\delta n} \gr 0.
\end{align*}
We thus deduce that there exists some $U_0 \in O(n)$ such that, if we set $x_i=U_0(e_i)$, $i=1,\ldots,n$, then
\[
\left\|\sum_{i=1}^n \epsilon_i x_i \right\|_\infty \gr c\sqrt{\log (e/\delta )},
\]
for any $\epsilon=(\epsilon_1,\ldots,\epsilon_n)\in S$, given that $1/n<\delta<1/4$.

Finally, note that in the case $\delta\in(1/4,1)$ the wanted result holds trivially, since for any $U\in O(n)$ and every $\epsilon\in \{-1,1\}^n$,
\[
\left\| \sum_{i=1}^n \epsilon_iUe_i \right\|_\infty \gr 1 \gr c\sqrt{\log (e/\delta)}
\]
is valid, for a suitable absolute constant $c>0$.
\end{proof}

\section{Signed sums of random vectors from a convex body}\label{sec:random.rot}

The argument used for the proof of Theorem \ref{intro-thm.1} motivates us to consider a more general setting. As a first step, we let the $\ell_{\infty}^n$ norm be replaced by an arbitrary norm in $\mathbb{R}^n$ and relax our assumptions on the choice of the vectors $x_1,\ldots,x_n\in B_2^n$. As a second step, we will consider a further generalization, letting $x_1,\ldots,x_n$ be chosen uniformly and independently from $B_2^n$, or any convex body $K$.

\subsection{Norms of signed sums of vectors}\label{subsec:3.1}

Let $D$ be a symmetric convex body in ${\mathbb R}^n$.  We consider the median
$m(D)$ of $\|Z\|_D$ where $Z\sim N(0,I_n)$ is a standard Gaussian random vector in ${\mathbb R}^n$.
Repeating the argument of the previous section we can show the following.

\begin{proposition}\label{prop:relaxed}Let $D$ be a symmetric convex body in ${\mathbb R}^n$ and $\tau >0$. Assume that the
vectors $x_1,\ldots ,x_n\in \mathbb{R}^n$ satisfy
\[
\left\|\sum_{i=1}^n \zeta_ix_i\right\|_2 \gr \tau\sqrt{n},
\]
for all $\zeta_i=\pm 1$. Then, for any $S\subseteq\{-1,1\}^n$ with
$$|S|\cdot \gamma_n\Big (\frac{2t}{\tau }m(D)D\Big )<1/2,$$
we may find ${\cal U}\subseteq O(n)$ with $\nu_n ({\cal U})>1- 2|S|\cdot \gamma_n\left (\frac{2t}{\tau }m(D)D\right )$
such that any $U\in {\cal U}$ satisfies
$$\left\|\sum_{i=1}^n\epsilon_iU(x_i)\right\|_D\gr t\cdot m(D)$$
for all $\epsilon =(\epsilon_1,\ldots ,\epsilon_n)\in S$.
\end{proposition}

\begin{proof}
Let $x_1,\ldots ,x_n\in \mathbb{R}^n$ be as in the statement above, and fix $\epsilon_1,\ldots,\epsilon_n$. Normalizing by $\left\|\sum_{i=1}^n \epsilon_ix_i\right\|_2$ and using the fact that the latter is greater than $\tau\sqrt{n}$, we get
\begin{align*}
& \nu_{n} \left( \left\{ U\in O(n) : \left\|\sum_{i=1}^n \epsilon_i U(x_i)\right\|_D \ls t\cdot m(D) \right\}\right) \\
&\hspace*{3cm} \ls \nu_{n} \left( \left\{ U\in O(n) : \left\|U\left( \frac{\sum_{i=1}^n \epsilon_ix_i}{\left\|\sum_{i=1}^n \epsilon_ix_i\right\|_2}\right)\right\|_D
\ls \frac{t\cdot m(D)}{\tau\sqrt{n}}\right\}\right)\\
&\hspace*{3cm} = \sigma\left(\left\{ \theta\in S^{n-1} : \|\theta\|_D \ls \frac{t\cdot m(D)}{\tau\sqrt{n}} \right\}\right)\\
&\hspace*{3cm} \ls 2\gamma_n\Big (\frac{2t}{\tau }m(D)D\Big ),
\end{align*}
where in the last step, we make use of \eqref{eq.sigma.to.gamma}. Therefore,
\begin{align*}
\nu_n &\left( \left\{ U\in O(n) : \left\| \sum_{i=1}^n \epsilon_i U(x_i)\right\|_D \gr t\cdot m(D) \hbox{\,, for every } \epsilon\in S \right\} \right) \\
&\hspace{100pt}\gr 1 - \sum_{\epsilon\in S}\nu_n\left(\left\{ U\in O(n) : \left\|\sum_{i=1}^n \epsilon_i U(x_i)\right\|_D \ls t\cdot m(D) \right\}\right) \\
&\hspace{100pt}> 1 - 2|S|\cdot \gamma_n\left (\frac{2t}{\tau }m(D)D\right ).
\end{align*}
This proves the claim of the proposition.
\end{proof}

Theorem \ref{intro-thm.2} now is a direct corollary of Proposition \ref{prop:relaxed}.

\begin{proof}[Proof of Theorem \ref{intro-thm.2}]
Let $t=\tau t_{D,\delta}$ and $S\subseteq\{-1,1\}^n$ with $|S|\ls 2^{\delta n}$.
By the definition of $t_{D,\delta}$ we have that $|S|\gamma_n(\frac{2t}{\tau}m(D)D)< e^{-n}$.
We can thus apply Proposition \ref{prop:relaxed} to get ${\cal U}\subseteq O(n)$ such that $\nu_n({\cal U})\gr 1-e^{-n}$ and
\[
\left\|\sum_{i=1}^n \epsilon_iUx_i\right\|_D \gr t\cdot m(D) =\tau t_{D,\delta}m(D)
\]
for every $U\in {\cal U}$ and all $(\epsilon_1,\ldots ,\epsilon_n)\in S$, which is the statement of the theorem.
\end{proof}

Before we proceed, let us briefly explain at this point the general bounds on the parameter $t_{D,\delta}$, stated in the introduction.

\begin{proof}[Proof of \eqref{eq.t_D.bounds}]
For the upper bound, it is straightforward by the definition of the median that
\[
\gamma_n(m(D)D) = \gamma_n(\|Z\|_D\ls m(D)) \gr \frac{1}{2} > (2^\delta e)^{-n},
\]
hence $t_{D,\delta}\ls \frac{1}{2}$. For the lower bound, since, integrating in polar coordinates, $\int_{S^{n-1}} \|x\|_D^{-n} \,d\sigma(x) = \frac{|D|}{|B_2^n|}$, Markov's inequality implies that
\[
\sigma\left( \|x\|_D \ls e^{-\eta}\left(\frac{|B_2^n|}{|D|} \right)^{1/n}\right) \ls e^{-\eta n}
\]
holds for every $\eta>0$. We can relate the latter to the Gaussian measure; using the same argument as in the proof of \cite[Lemma 2.1]{Klartag-Vershynin-2007}, one can prove that for any $a\gr 1$,
\[
\gamma_n\left(\frac{1}{a}\sqrt{n}e^{-\eta}\left(\frac{|B_2^n|}{|D|} \right)^{1/n}D\right) \ls \sigma\left( \|x\|_D \ls e^{-\eta}\left(\frac{|B_2^n|}{|D|} \right)^{1/n}\right) +\gamma_n\left(\frac{1}{a}\sqrt{n}B_2^n\right) \ls e^{-\eta n} +\gamma_n\left(\frac{1}{a}\sqrt{n}B_2^n\right).
\]
We can bound the second term using, e.g. \cite[Proposition 2.2]{Barvinok}, to get $\gamma_n\left(\frac{1}{a}\sqrt{n}B_2^n\right) \ls a^{-n}\exp\left(\frac{n(a^2-1)}{2a^2}\right)$. Note that, for $\eta=\log(4e)$ we have $2e^{-\eta n}=2^{-n+1}(2e)^{-n}\ls (2^\delta e)^{-n}$ for any $\delta\in(0,1)$, so it suffices to have
\[
a^{-n}\exp\left(\frac{n(a^2-1)}{2a^2}\right) \ls (4e)^{-n} = e^{-\eta n},
\]
which is satisfied if $a$ is large enough ($a=20$ will do). Now, since $ \gamma_n\left(\frac{1}{80e}\sqrt{n}\left(\frac{|B_2^n|}{|D|} \right)^{1/n}D\right) \ls (2^\delta e)^{-n}$ for every $\delta\in(0,1)$, the definition of $t_{D,\delta}$ implies that
\[
t_{D,\delta} \gr \frac{1}{160e}\frac{\sqrt{n}}{m(D)}\left(\frac{|B_2^n|}{|D|} \right)^{1/n} \gr c_1 \frac{1}{|D|^{1/n}m(D)},
\]
for some absolute constant $c_1>0$.
\end{proof}

\subsection{Random points from convex bodies}

We will see now that, suitably normalized, the hypothesis on the norm of the sum $\zeta_1x_1+\cdots +\zeta_nx_n$ in the statement of
Proposition \ref{prop:relaxed} is satisfied with overwhelming probability when the $x_i$'s are chosen uniformly and independently at
random from the interior of a general symmetric convex body $K$. The required lower bound, which actually holds for any norm in
$\mathbb{R}^n$ in place of the Euclidean norm, will be deduced as a corollary of the following result of
Gluskin and V. Milman \cite{Gluskin-Milman-2004}:

\begin{proposition}[Gluskin-V.~Milman]\label{prop:gluskin}Let $D$ be a star body in ${\mathbb R}^n$ with $0\in {\rm int}(D)$
and $V_1,\ldots ,V_m$ be measurable subsets of ${\mathbb R}^n$ with $|D|=|V_1|=\cdots =|V_m|$. For any $\lambda_1,\ldots ,\lambda_m\in {\mathbb R}$
and any $0<t<1$ one has
\[
\mathbb{P}\left( (x_i)_{i=1}^m, x_i\in V_i: \left\|\sum_{i=1}^m \lambda_ix_i \right\|_D \ls t\Big (\sum_{i=1}^m\lambda_i^2\Big )^{1/2}\right) \ls
\left (te^{\frac{1-t^2}{2}}\right )^n .
\]
\end{proposition}

Below, we will make use of the following special instance of Proposition \ref{prop:gluskin}, which also provides an alternative proof
of Banaszczyk's general lower bound \eqref{eq:ban-lower} for $\beta (K,D)$.

\begin{corollary}\label{corol.signed.norm.lower}
Let $K$ and $D$ be symmetric convex bodies in $\mathbb{R}^n$, and let $x_1,\ldots,x_n$ be random points chosen uniformly from $K$. The inequality
\[
\left\|\sum_{i=1}^n \epsilon_ix_i \right\|_D \gr \frac{1}{10}\left(\frac{|K|}{|D|} \right)^{1/n}\sqrt{n}
\]
is then valid for any choice of $\epsilon_1,\ldots,\epsilon_n\in\{-1,1\}$, with probability greater than $1-e^{-n}$ with respect to $x_1,\ldots,x_n$.
\end{corollary}

\begin{proof}Let $t>0$ and assume first that $|K|=|D|$. We have
\begin{align*}
&\mathbb{P}\left( (x_i)_{i=1}^n \subseteq K : \left\|\sum_{i=1}^n \frac{1}{\sqrt{n}}\epsilon_ix_i \right\|_D \ls  t, \hbox{ for some } \epsilon_1,\ldots,\epsilon_n \in \{-1,1\} \right)\\
&\hspace{200pt}\ls 2^n \mathbb{P}\left( (x_i)_{i=1}^n \subseteq K : \left\|\sum_{i=1}^n \frac{1}{\sqrt{n}}\epsilon_ix_i \right\|_D \ls  t\right)
\end{align*}
Now if we apply Proposition \ref{prop:gluskin} for $m:=n$, $V_i:=K$ and $\lambda_i:=\frac{1}{\sqrt{n}}\epsilon_i$ for every $i$,
and $t$ such that $2te^{(1-t^2)/2}< e^{-1}$ (say $t=1/10$) we get
\[
\mathbb{P}\left( (x_i)_{i=1}^n \subseteq K : \left\|\sum_{i=1}^n \frac{1}{\sqrt{n}}\epsilon_ix_i \right\|_D \ls  t, \hbox{ for some } \epsilon_1,\ldots,\epsilon_n \in \{-1,1\} \right) \ls 2^n(te^{\frac{1-t^2}{2}})^n < e^{-n}.
\]
We deduce that, with probability greater than $1-e^{-n}$ with respect to $(x_1,\ldots ,x_n)$, we have
\[
\left\| \sum_{i=1}^n \epsilon_ix_i\right\|_D > \frac{1}{10}\sqrt{n}
\]
for every choice of $\epsilon_1,\ldots,\epsilon_n \in \{-1,1\}$.

For a general pair of symmetric convex bodies $K,D$, set $a=(|D|/|K|)^{1/n}$ and $\tilde{K} = aK$. Then we can apply the above for
the pair $\tilde{K}, D$ to deduce that for every choice of $\epsilon_1,\ldots,\epsilon_n \in \{-1,1\}$,
\[
\left\| \sum_{i=1}^n \epsilon_iax_i\right\|_D > \frac{1}{10}\sqrt{n},
\]
or, equivalently,
\[
\left\| \sum_{i=1}^n \epsilon_ix_i\right\|_D > \frac{1}{10}\sqrt{n}\left(\frac{|K|}{|D|}\right)^{1/n}
\]
holds with probability greater than $1-e^{-n}$ with respect to $(x_1,\ldots ,x_n)$.
\end{proof}

\begin{remark}\rm
Note that Corollary \ref{corol.signed.norm.lower} immediately implies  Banaszczyk's lower bound on $\beta(K,D)$:
\[
\beta (K,D)\gr c\sqrt{n}(|K|/|D|)^{1/n}.
\]
\end{remark}

\medskip

\subsubsection{Points from the ball}

First we deal with the situation where the vectors $x_1,\ldots,x_n$ are chosen independently and uniformly at random from $B_2^n$. The statement of Theorem \ref{general-thm} follows immediately from the definition of $t_{D,\delta}$ and the next consequence of Proposition \ref{prop:relaxed}. This can be viewed as a generalization of Hajela's result, in the spirit of Proposition \ref{prop:gluskin}.

\begin{theorem}\label{thm.4}
Let $D$ be a symmetric convex body in $\mathbb{R}^n$ and $S\subseteq \{-1,1\}^n$. Then, for every $t>0$,
\[
\mathbb{P}\left( (x_i)_{i=1}^n\subseteq B_2^n : \left\|\sum_{i=1}^n \epsilon_ix_i\right\|_D \ls  t\cdot m(D) \hbox{, for some } \epsilon\in S\right)
\ls 2|S|\cdot\gamma_n(20t\cdot m(D)D)+e^{-n}.
\]
\end{theorem}

\begin{proof}
Let $A\subseteq (B_2^n)^n$ be the set
\[
A := \left\{ (x_i)_{i=1}^n \subseteq B_2^n : \left\|\sum\epsilon_ix_i\right\|_2\gr \frac{1}{10}\sqrt{n} \hbox{, for every } \epsilon_1,\ldots,\epsilon_n\in\{-1,1\}\right\}.
\]
By Corollary \ref{corol.signed.norm.lower}, applied for $K=D=B_2^n$, we have that $\mathbb{P}(A^c)< e^{-n}$. This fact, combined with the same reasoning as in the proof of Proposition \ref{prop:relaxed} gives
\begin{align*}
&\mathbb{P}\left( (z_i)_{i=1}^n \subseteq B_2^n : \left\| \sum_{i=1}^n \epsilon_iz_i \right\|_D \ls t\cdot m(D) \hbox{, for some } \epsilon\in S \right)\\
&\hspace{30pt}= \mathbb{P}\left(\left((x_i)_{i=1}^n,U\right)\in (B_2^n)^n\times O(n) : \left\| \sum_{i=1}^n \epsilon_iU(x_i) \right\|_D \ls t\cdot m(D) \hbox{, for some } \epsilon\in S \right) \\
&\hspace{30pt}\ls \mathbb{P}\left(\left((x_i)_{i=1}^n,U\right)\in A\times O(n) : \left\| \sum_{i=1}^n \epsilon_iU(x_i) \right\|_D \ls t\cdot m(D) \hbox{, for some } \epsilon\in S\right) + e^{-n}\\
&\hspace{30pt}= \int_A\left [\nu_n\left(\left\{U\in O(n) : \left\|\sum_{i=1}^n \epsilon_iU(x_i)\right\|_D \ls t\cdot m(D)
\hbox{, for some } \epsilon\in S \right\}\right)\right ]\,d\mu (x_n)\cdots d\mu (x_1) + e^{-n}\\
&\hspace{30pt}< 2|S|\gamma_n(20t\cdot m(D)D)+e^{-n},
\end{align*}
as claimed.
\end{proof}

\paragraph{Application: the case of $\ell_p^n$.}

As an application let us consider the case $D=B_p^n$, $1\ls p\ls\infty $. We remark that, although the estimate $\beta(B_2^n, B_2^n)\ls \sqrt{n}$ seems to be well-known
(a proof of this can be found in \cite{Bar}), we could not locate in the literature any upper bound on the parameter $\beta(B_2^n,B_p^n)$, for $p\neq 2$.
However, for $1\ls p< 2$, one can use the previously mentioned estimate for $\beta(B_2^n, B_2^n)$: if we know that for every $x_1,\ldots, x_n\in B_2^n$ there exists $(\epsilon_i)_{i=1}^n\subseteq \{-1,1\}$ such that $\|\epsilon_1x_1+\ldots+\epsilon_nx_n\|_2\ls \sqrt{n}$, then one can write
\[
\|\epsilon_1x_1+\ldots+\epsilon_nx_n\|_p \ls n^{\frac{1}{p}-\frac{1}{2}}\|\epsilon_1x_1+\ldots+\epsilon_nx_n\|_2 \ls n^{1/p}.
\]
The above estimate is actually sharp, by the lower bound \eqref{eq:ban-lower} and the fact that $(|B_2^n|/|B_p^n|)^{1/n}$ is of the order of $n^{\frac{1}{p}-\frac{1}{2}}$. On the other hand, for the case $p>2$ one could use the estimate of Banasczcyk $\beta(B_2^n,B_\infty^n)\ls c\sqrt{\log n}$ in a similar fashion, to deduce that $\beta(B_2^n,B_p^n) \ls cn^{1/p}\sqrt{\log n}$.

Recall the definition of the parameter $\beta$ in \cite{Paouris-Valettas-Zinn-2017}:
\[
\beta (D)=\frac{{\rm Var}_{\gamma_n}(\|\cdot\|_D)}{M^2(D)}.
\]
For $D=B_p^n$, $\beta (D)$ has been computed
in \cite{Paouris-Valettas-Zinn-2017}: one has $\beta (B_p^n)\simeq \frac{2^p}{p^2n}$ if $1\ls p\ls c\log n$ and
$\beta (B_p^n)\simeq (\log n)^{-2}$ if $C\log n\ls p\ls\infty $, for some absolute constants $c, C >0$ (for a more detailed analysis, one may consult \cite{Lytova-Tikhomirov-2017}). Moreover it is known that, in general,
$m(D)\simeq {\mathbb E}\,\|Z\|_D$ for any symmetric convex body $D$ in ${\mathbb R}^n$ (the inequality $\mathbb{E}\|Z\|_D\ls c m(D)$ can be deduced from \cite[Lemma 3.1]{LedTal}, while $m(D)\ls 2\mathbb{E}\|Z\|_D$ by the definition of $m(D)$ and Markov's inequality), and in particular
$$m(B_p^n)\simeq {\mathbb E}\,\|Z\|_p\simeq n^{1/p}\sqrt{p},\qquad 1\ls p\ls\log n,$$
while
$$m(B_p^n)\simeq {\mathbb E}\,\|Z\|_p\simeq \sqrt{\log n},\qquad \log n\ls p\ls\infty .$$
Therefore, choosing $t=\frac{1}{4}$ in Theorem \ref{thm.4} and using the estimate $\gamma_n(\{x:\|x\|_p\ls t\cdot m(B_p^n)\})\ls\frac{1}{2}t^{c/\beta (B_p^n)}$, we get:

\begin{corollary}\label{cor.3}For any $p\gr 1$ there exists a constant $c_p>0$ such that for any $n\gr n_0(p)$
and any $S\subseteq\{-1,1\}^n$ with $|S|\ls 2^{c_pn}$ we have that a random $n$-tuple of points in $B_2^n$ satisfies, with
probability greater than $1-e^{-n}$,
$$\left\|\sum_{i=1}^n\epsilon_ix_i\right\|_p\gr c\sqrt{p}n^{1/p}$$
for all $\epsilon =(\epsilon_1,\ldots ,\epsilon_n)\in S$.
\end{corollary}

\begin{remark}\rm
It is useful to compare the last result with \eqref{eq:ban-lower}. For every $p\ls \log n$, since $(|B_2^n|/|B_p^n|)^{1/n}$ is of the order of $n^{\frac{1}{p}-\frac{1}{2}}$ we see that for any $S\subseteq\{-1,1\}^n$ with $|S|\ls 2^{c_pn}$ a random $n$-tuple of points in $B_2^n$ satisfies, with
probability greater than $1-e^{-n}$,
$$\left\|\sum_{i=1}^n\epsilon_ix_i\right\|_p\gr c\sqrt{p}\sqrt{n}\left (\frac{|B_2^n|}{|B_p^n|}\right )^{1/n}$$
for all $\epsilon =(\epsilon_1,\ldots ,\epsilon_n)\in S$. On the other hand, in the case $p>\log n$ one can use the fact that  the norms $\|\cdot\|_p$ and $\|\cdot\|_\infty$ are equivalent to deduce from Theorem \ref{intro-thm.1} (more precisely, Theorem \ref{general-thm} combined with the estimate given by Theorem \ref{intro-thm.1}) that
for any $0<\varrho <1$ and $S\subseteq\{-1,1\}^n$ with $|S|\ls 2^{n^{1-\varrho }}$, a random $n$-tuple of points in $B_2^n$ satisfies, with
probability greater than $1-e^{-n}$,
$$\left\|\sum_{i=1}^n\epsilon_ix_i\right\|_p\gr c(\varrho )\sqrt{\log n}\sqrt{n}\left (\frac{|B_2^n|}{|B_p^n|}\right )^{1/n}$$
for all $\epsilon =(\epsilon_1,\ldots ,\epsilon_n)\in S$.

In fact, the results of this section point out a general way to derive variants of \eqref{eq:ban-lower} in this spirit.
A meaningful lower bound on $\|\sum_{i=1}^n \epsilon_i x_i\|_D$ can be obtained with high probability for the random $n$-tuple
$x_1,\ldots,x_n$, if one is willing to drop the requirement that this holds for any
$\epsilon=(\epsilon_1,\ldots,\epsilon_n)\in \{-1,1\}^n$, but rather restrict themselves to a ``big'' subset $S\subseteq \{-1,1\}^n$.
\end{remark}

\subsubsection{Points from a symmetric convex body}

Finally, we study the case where $x_1,\ldots,x_n$ are chosen uniformly and independently from an arbitrary symmetric convex body $K$
in ${\mathbb R}^n$. We shall prove the following generalization of Theorem \ref{thm.4}, which in turn, for $t:= ct_{D,\delta}$ will give the claim in Theorem \ref{intro-thm.5}.

\begin{theorem}\label{thm.5}
There exists an absolute constant $c>0$ such the following holds: Let $D$ be a symmetric convex body in $\mathbb{R}^n$ and $S\subseteq \{-1,1\}^n$. For any
symmetric convex body $K$ in ${\mathbb R}^n$ let $t>0$ be such that
\begin{equation*}
|S|\gamma_n(ct(|B_2^n|/|K|)^{1/n}m(D)D)<e^{-n}.
\end{equation*}
Then, we may find ${\cal U}\subseteq O(n)$ with $\nu_n ({\cal U})>1-3e^{-n/2}$
such that, for all $U\in {\cal U}$,
\begin{equation*}\mathbb{P}\left( (z_i)_{i=1}^n \subseteq U(K)\times\cdots\times U(K) : \left\| \sum_{i=1}^n \epsilon_iz_i \right\|_D \ls t\cdot m(D) \hbox{,
for some } \epsilon\in S \right)
\ls e^{-n/2}.\end{equation*}
\end{theorem}

\begin{proof}
Let
\[
A = \left\{ (x_i)_{i=1}^n\subseteq K : \left\|\sum_{i=1}^n \epsilon_ix_i \right\|_2 > \frac{1}{10}\left(\frac{|K|}{|B_2^n|} \right)^{1/n}\sqrt{n}, \hbox{for every } \epsilon_1,\ldots,\epsilon_n \in \{-1,1\}\right\}.
\]
By Corollary \ref{corol.signed.norm.lower}, applied for $D=B_2^n$, we have that $\mathbb{P}(A^c)< e^{-n}$. Using Proposition \ref{prop:relaxed} we write
\begin{align*}
&\int_{O(n)}\mathbb{P}\left( (z_i)_{i=1}^n \in U(K)^n : \left\| \sum_{i=1}^n \epsilon_iz_i \right\|_D \ls t\cdot m(D) \hbox{, for some } \epsilon\in S \right)\\
&\hspace{30pt}= \mathbb{P}\left(\left((x_i)_{i=1}^n,U\right)\in K^n\times O(n) : \left\| \sum_{i=1}^n \epsilon_iUx_i \right\|_D \ls t\cdot m(D) \hbox{, for some } \epsilon\in S \right) \\
&\hspace{30pt}\ls \mathbb{P}\left(\left((x_i)_{i=1}^n,U\right)\in A\times O(n) : \left\| \sum_{i=1}^n \epsilon_iUx_i \right\|_D \ls t\cdot m(D) \hbox{, for some } \epsilon\in S\right) + e^{-n}\\
&\hspace{30pt}= \int_A\left [\nu_n\left(\left\{U\in O(n) : \left\|\sum_{i=1}^n \epsilon_iUx_i\right\|_D \ls t\cdot m(D)
\hbox{, for some } \epsilon\in S \right\}\right)\right ]\,d\mu (x_n)\cdots d\mu (x_1) + e^{-n}\\
&\hspace{30pt}< 2|S|\gamma_n(20t(|B_2^n|/|K|)^{1/n}m(D)D)+e^{-n}.
\end{align*}
Assume that $t$ is chosen so that
\begin{equation*}|S|\gamma_n(20t(|B_2^n|/|K|)^{1/n}m(D)D)<e^{-n}.\end{equation*}
Applying Markov's inequality we may find ${\cal U}\subseteq O(n)$ with $\nu_n ({\cal U})>1-3e^{-n/2}$
such that
\begin{equation*}\mathbb{P}\left( (z_i)_{i=1}^n \subseteq U(K)^n : \left\| \sum_{i=1}^n \epsilon_iz_i \right\|_D \ls t\cdot m(D) \hbox{, for some } \epsilon\in S \right)
\ls e^{-n/2}.\end{equation*}
This proves the theorem.\end{proof}

\begin{remark}\rm
Theorem \ref{thm.5} illustrates once more the main idea behind the improvement upon Hajela's result, Theorem \ref{intro-thm.1}, as well as the rest of the results in this note; small ball probability estimates for the Gaussian measure of convex bodies can be linked to the deduction of lower bounds for variants of the quantity $\beta(K,D)$, where the $D$-norm of the signed sum of a random $n$-tuple of vectors can be lower bounded for every choice of signs in any appropriately large subset of $\{-1,1\}^n$.
\end{remark}

\medskip

\noindent {\bf Acknowledgements.} We would like to thank Apostolos Giannopoulos for useful discussions and the anonymous referees for their comments and valuable suggestions that helped to improve the presentation of the results of this article. The research of the first named author is supported by the National Scholarship Foundation (IKY), sponsored by the act ``Scholarship grants for second-degree graduate studies'', from resources of the operational program ``Manpower Development, Education and Life-long Learning'', 2014-2020, co-funded by the European Social Fund (ESF) and the Greek state. The second named author is supported by a PhD Scholarship from the Hellenic Foundation for Research and Innovation (ELIDEK); research number 70/3/14547.

\footnotesize
\bibliographystyle{amsplain}

\bigskip

\bigskip

\bigskip

\noindent{\bf Keywords:} convex bodies, balancing vectors, signed sums, Koml\'{o}s conjecture, random rotations, gaussian measure, small ball probability.
\\
\thanks{\noindent {\bf 2010 MSC:} Primary 52A23; Secondary 46B06, 52A40, 60D05.}

\bigskip

\bigskip

\noindent \textsc{Giorgos \ Chasapis}: Department of
Mathematics, University of Athens, Panepistimioupolis 157-84,
Athens, Greece.

\smallskip 

\noindent {\it Current Address:} Department of Mathematical Sciences, Kent State University, Kent, OH USA.

\smallskip

\noindent \textit{E-mail:} \texttt{gchasap1@kent.edu}

\bigskip

\noindent \textsc{Nikos \ Skarmogiannis}: Department of
Mathematics, University of Athens, Panepistimioupolis 157-84,
Athens, Greece.

\smallskip

\noindent \textit{E-mail:} \texttt{nikskarm@math.uoa.gr}

\bigskip

\end{document}